\pdfoutput=1    

\documentclass[11pt,a4paper]{article} 

\usepackage[OT2,T1]{fontenc}
\usepackage{amsmath}
\usepackage{amssymb,verbatim}
\usepackage{amsfonts}
\usepackage{amsthm}
\usepackage{pifont}
\usepackage{graphicx}
\usepackage{datetime}
\usepackage{scrdate}
\usepackage{enumitem}
\usepackage{geometry}
\usepackage{float}
\usepackage{bigints}
\usepackage{multirow}

\usepackage{hyperref}

\hypersetup{colorlinks=true, linkcolor=red, citecolor=red, pdfstartview=FitH, linktocpage=false}

\newtheorem{thm}{Theorem}[section]
\newtheorem{lem}[thm]{Lemma}
\newtheorem{cor}[thm]{Corollary}

\theoremstyle{definition}

\newtheorem{example}[thm]{\bs\;Example}

\numberwithin{equation}{section}

\newcommand{\beq}{\begin{equation*}}
\newcommand{\eeq}{\end{equation*}}
\newcommand{\beqn}{\begin{equation}}
\newcommand{\eeqn}{\end{equation}}
\newcommand{\dd}{\mathrm{d}}

\newcommand{\tn}{\textnormal}
\newcommand{\ds}{\displaystyle}

\newcommand{\ph}{\phi}
\newcommand{\eps}{\varepsilon}
\newcommand{\mI}{m_\tn{i}}
\newcommand{\moi}{m_\tn{i}^{01}}
\newcommand{\mio}{m_\tn{i}^{10}}
\newcommand{\AI}{A_\tn{i}}
\newcommand{\Aoi}{A_\tn{i}^{01}}
\newcommand{\Aio}{A_\tn{i}^{10}}
\newcommand{\Ft}{\widetilde{F}}
\newcommand{\pI}{p_\tn{i}}

\newcommand{\pe}{p_\tn{e}}
\newcommand{\Cp}{C_1^+}
\newcommand{\Cm}{C_1^-}
\newcommand{\bs}{$\bigstar$}

\begin{document}

\title{Measures and geometric probabilities for ellipses\\ intersecting circles}
\author{Uwe B\"asel}
\date{} 
\maketitle
\begin{abstract}
\noindent Santal\'o calculated the measures for all positions of a moving line segment in which it lies inside a fixed circle and intersects this circle in one or two points. From these measures he concluded hitting probabilities for a line segment thrown randomly onto an unbounded lattice of circles. In the present paper these results are generalized to ellipses  instead of line segments. The respective measures for all positions of a moving ellipse in which it lies completely inside a fixed circle, encloses it, and intersects it in two or four points are derived. Then the hitting probabilities for lattices of circles are deduced. It is shown that the results for a line segment follow as special cases from those of the ellipse.\\[0.2cm]
\end{abstract}

\section{Introduction}

Santal\'o \cite{Santalo40} calculated the measure $\mathfrak{M}_\tn{i}$ of all oriented line segments of length $\ell$ that are completely contained in a disk of radius $r$. If $\ell\ge 2r$, $\mathfrak{M}_\tn{i}$ is obviously equal to zero; if $\ell\le 2r$, then
\begin{align} \label{Eq:Mi}
  \mathfrak{M}_\tn{i}
= {} & 2\pi\left(\pi r^2 - 2r^2\arcsin\frac{\ell}{2r}
- \ell\,\sqrt{r^2-\frac{\ell^2}{4}}\,\right).
\end{align}
From this he concluded the measures
\beqn \label{Eq:M1_M2} 
  \mathfrak{M}_1 = 4\pi^2r^2 - 2\hspace{1pt}\mathfrak{M}_\tn{i}\,,\quad
  \mathfrak{M}_2 = 4\pi r\ell - 2\pi^2r^2 + \mathfrak{M}_\tn{i}
\eeqn
for all line segments intersecting the circle in one and two points, respectively. Then he considered the random throw of a line segment onto an unbounded lattice of circles, as shown as an example in Fig.\ \ref{Fig:Reseau}. Such a lattice consists of circles of radius $r$ whose center points are placed at the vertices of parallelograms with sides of length $s$ and $t$, and angle $\sigma$. Under the assumption that the line segment can hit only one circle of the lattice at the same time, he derived the geometric probabilities that this line segment, placed randomly onto the lattice, hits a circle in one point, in two points, lies completely inside a circle, and outside all circles (see Fig.\ \ref{Fig:Reseau}).

Duma and Stoka calculated the probability that an ellipse hits a lattice of parallelograms \cite{Duma_Stoka}.

Most recently, B\"ottcher calculated the measures and geometric probabilities that an arbitrarily long line segment hits exactly one or two sides of a triangle \cite{Boettcher01}, and that this segment hits two non-overlapping circles \cite{Boettcher02}.

In this paper we generalize Santal\'o's above mentioned results for ellipses (see Fig.\ \ref{Fig:Reseau}) instead of line segments. 

\begin{figure}[h]
  \begin{center}
	\includegraphics[scale=0.7]{
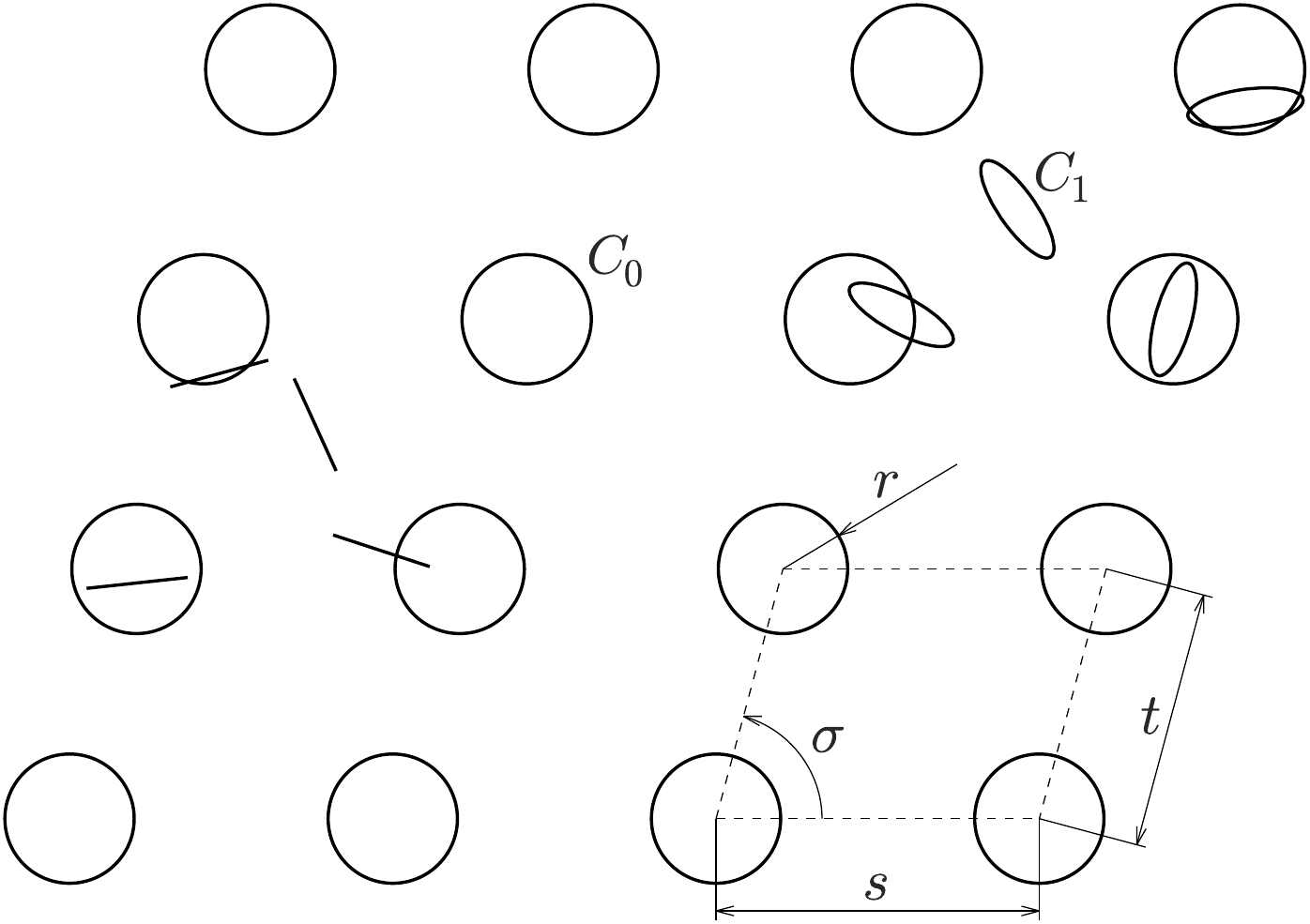}
  \end{center}
  \vspace{-0.4cm}
  \caption{\label{Fig:Reseau} Random throws of an ellipse $C_1$ and a line segment onto a lattice of circles $C_0$}
\end{figure}

We denote by $K_0$ the closed disk bounded by the circle $C_0:=\partial K_0$ of radius $r$, and by $K_1$ the closed set of points bounded by the ellipse $C_1:=\partial K_1$ with semi-major axis of length $a$ and semi-minor axis of length $b$. 

\newpage\noindent
The present paper is organized as follows:
\begin{itemize}[leftmargin=0.6cm] \setlength{\itemsep}{-1pt}
\item In Section \ref{Sec:Intersections} we first consider a moving ellipse $C_1^*$ with fixed direction (indicated by the asterisk) and a fixed circle $C_0$, and discuss the possible intersection cases depending on the position of the center point $M_1$ of $C_1^*$. We show that it is possible to consider the inverse translation with $C_0$ moving and $C_1^*$ fixed instead of the original motion. Then, the outer and the inner parallel curve of $C_1^*$ in the distance $r$ ($=$ radius of $C_0$) bound sets which are essential for the following investigations.       
\item In Section \ref{Sec:Measures1} we calculate areas which are the measures for all positions of $C_1^*$ with $C_0\subset K_1^*$ or $C_1^*\subset K_0$, and all positions of $C_1^*$ in which it intersects $C_0$ in two or four points.
\item In Section \ref{Sec:Measures2} we consider the motion of $C_1$ without the restriction to a fixed direction, and derive the respective measures for all positions of $C_1$ ($=$ all congruent copies of $C_1$) intersecting the fixed circle $C_0$.  
\item In Section \ref{Sec:Probabilities} the measures from Section \ref{Sec:Measures2} are used to find the hitting probabilities for an ellipse $C_1$ which is thrown randomly onto a lattice of circles $C_0$ as it is shown in Fig.\ \ref{Fig:Reseau}.
\item In Section \ref{Sec:Special_case} we show that Santal\'os measures \eqref{Eq:Mi}, \eqref{Eq:M1_M2} and probabilities for a line segment follow from the results in Sections \ref{Sec:Measures2} and \ref{Sec:Probabilities}, respectively.   
\end{itemize}  

\section{Intersections of a circle and an ellipse}
\label{Sec:Intersections}

From B\'ezout's theorem (see e.\ g.\ \cite[pp.\ 291-304]{Brieskorn_Knoerrer}, \cite[pp.\ 51-69]{Kirwan}) we know that an ellipse and a circle always have four intersection points if each point is counted with its intersection multiplicity. We have the following cases:

\begin{itemize}[leftmargin=0.6cm] \setlength{\itemsep}{-2pt}
\item[a)] All intersections points are real.
\item[b)] There are two real and two conjugate complex intersection points.
\item[c)] There are two pairs of conjugate complex intersection points.  
\end{itemize}
Even in the case that the ellipse is also a circle, B\'ezout's theorem holds true if homogeneous coordinates are used. Here two of the four intersection points are the so called {\em circular points at infinity}. All circles pass through this complex pair of points. (See e.\:g. \cite[p.\ 94]{Kowalewski} and \cite{Lamoen}.)

Here we are interested only in real intersection points.

\begin{example} \label{Exm:Example}
Let us consider the fixed circle $C_0$ and the moving ellipse $C_1$ with fixed direction in Fig.\ \ref{Fig:Intersections1}. We write $C_1^*$ in order to indicate that $C_1$ has fixed direction. $C_1^*$ intersects $C_0$ in two distinct points if the center point $M_1$ of $C_1^*$ lies in the open set bounded by the curves $C^*$ and $C^{**}$ (positions 1 and 2 of $C_1^*$). $C_1^*$ does not intersect $C_0$ if $M_1$ is outside $C^*$ (position 3) or inside the middle loop of $C^{**}$ (position 4). $C_1^*$ intersects $C_0$ in four distinct points if $M_1$ lies inside the upper or lower loop of $C^{**}$ (position 5). Clearly, the closed set bounded by $C^*$ is the Minkowski sum $K_0 + K_1^*$, where $K_1^*$ is the closed set bounded by $C_1^*$.

\begin{figure}[h]
  \begin{center}
	\includegraphics[scale=0.9]{
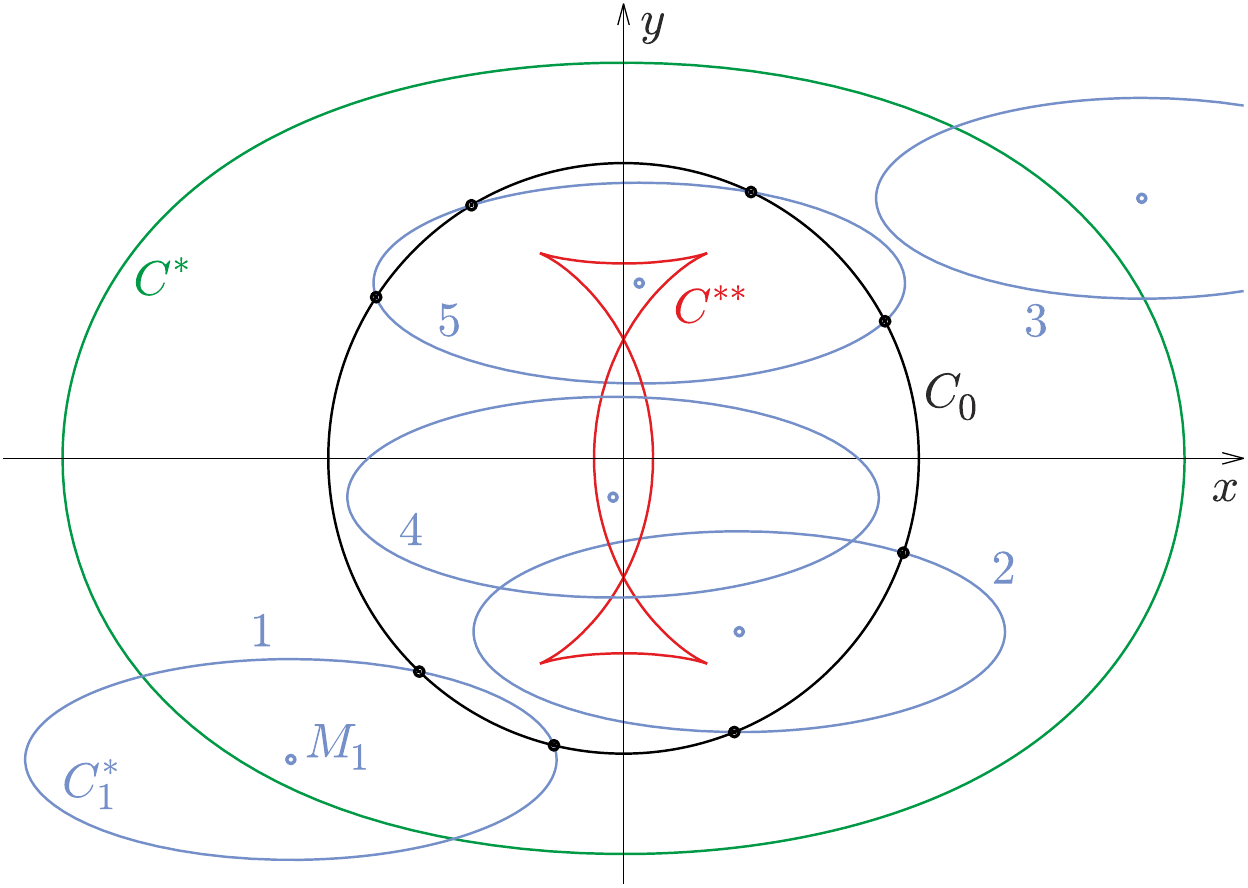}
  \end{center}
  \vspace{-0.5cm}
  \caption{\label{Fig:Intersections1} Fixed circle $C_0$, and moving ellipse $C_1^*$ in five positions $1,2,\ldots,5$}
\vspace*{0.5cm}
\end{figure}

\begin{figure}[h]
  \begin{center}
	\includegraphics[scale=0.9]{
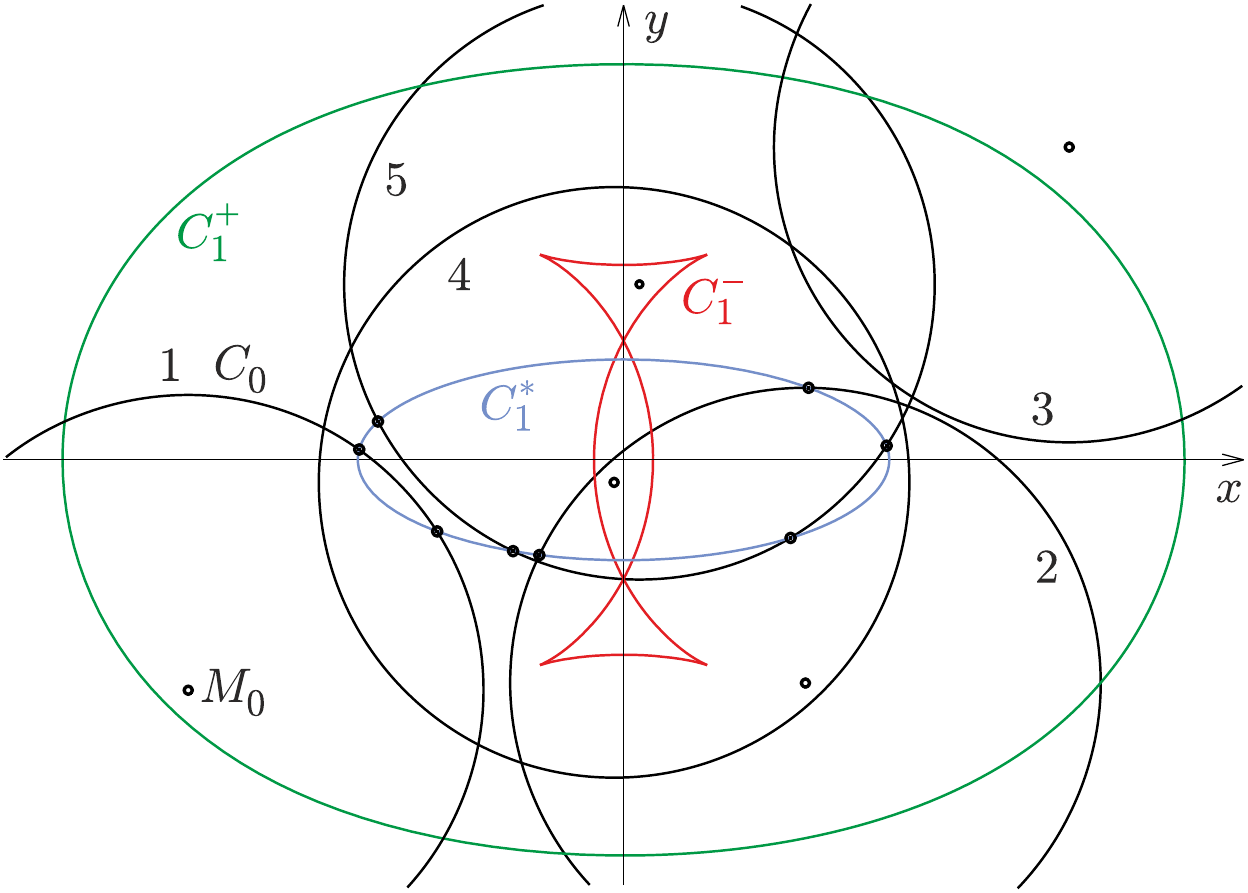}
  \end{center}
  \vspace{-0.5cm}
  \caption{\label{Fig:Intersections2} Fixed ellipse $C_1^*$, and moving circle $C_0$ in five positions $1,2,\ldots,5$}
\end{figure}

Now we consider the inverse translation which means that $C_1^*$ is fixed and $C_0$ is moving (see Fig.\ \ref{Fig:Intersections2}). Here we have the situation that $C_0$ intersects $C_1^*$ in two distinct points if the center point $M_0$ of $C_0$ lies in the open set bounded by the curves $\Cp$ and $\Cm$ (positions 1 and 2 of $C_0$). Due to the commutative property of the Minkowski addition, $K_0 + K_1^* = K_1^* + K_0 $, we immediately know that $\Cp \equiv C^*$. Since we consider only translations, we also have $\Cm \equiv C^{**}$. $\Cp$ and $\Cm$ are, respectively, the outer and inner parallel curve of $C_1^*$ in the distance $r$. Analogous to the original motion, one finds that $C_0$ does not intersect $C_1^*$ if $M_0$ is outside $\Cp$ (position 3) or inside the middle loop of $\Cm$ (position 4), and $C_0$ intersects $C_1^*$ in four distinct points if $M_0$ lies inside the upper or lower loop of $\Cm$ (position 5). \hfill\bs
\end{example}

\begin{figure}[h]
  \vspace{-0.3cm}
  \begin{center}
	\includegraphics[scale=0.87]{
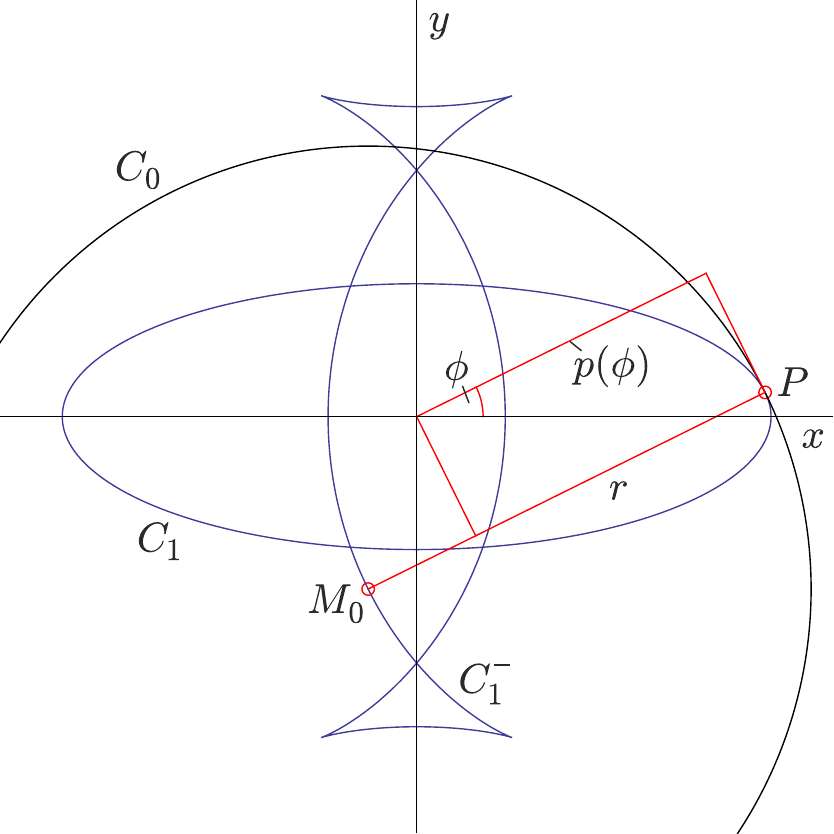}
  \end{center}
  \vspace{-0.6cm}
  \caption{\label{Fig:Support_function} Support functions $p(\phi)$ and $p(\ph)-r$ of $C_1$ and $\Cm$, respectively}
\end{figure}

Now, we will provide parametric representations of $C_1=C_1^*$, $\Cp$ and $\Cm$ with the support function $p(\ph)$ of $C_1$ (see Fig.\ \ref{Fig:Support_function}) that is given by
\beqn \label{Eq:support1}
  p(\ph)
= \sqrt{a^2\cos^2\ph + b^2\sin^2\ph}\,,\quad 0 \le \ph \le 2\pi\,,
\eeqn
with derivatives
\beq
  p'(\ph)
= - \frac{\left(a^2-b^2\right)\cos\ph\sin\ph}{\sqrt{a^2\cos^2\ph + b^2\sin^2\ph}}\,,\qquad
  p''(\ph)
= -\frac
	{\left(a^2-b^2\right)\left(a^2\cos^4\ph-b^2\sin^4\ph\right)}
	{\left(a^2\cos^2\ph + b^2\sin^2\ph\right)^{3/2}}\,.
\eeq  
Using the formulas
\beq
  x 
= p(\ph)\cos\ph - p'(\ph)\sin\ph\,,\qquad
  y 
= p(\ph)\sin\ph + p'(\ph)\cos\ph\,,
\eeq
(see e.\ g.\ \cite[p.\ 3]{Santalo}), we get 
\beq
  x
= \frac{a^2\cos\ph}{p(\ph)}\,,\quad
  y
= \frac{b^2\sin\ph}{p(\ph)}\,,\quad
  0 \le \ph \le 2\pi\,,
\eeq
as parametric parametric representation of $C_1$. $p(\ph)+r$ is the support function of $\Cp$, and $p(\ph)-r$ that of $\Cm$ (Fig.\ \ref{Fig:Support_function}.) Hence parametric representations of  $\Cp$ and $\Cm$ are given by
\beqn \label{Eq:Param_C1k}
\left.
\begin{aligned}
  x 
= {} & (p(\ph)+kr)\cos\ph - p'(\ph)\sin\ph 
= \left(\frac{a^2}{p(\ph)} + kr\right)\cos\ph\,,\\
  y 
= {} & (p(\ph)+kr)\sin\ph + p'(\ph)\cos\ph
= \left(\frac{b^2}{p(\ph)} + kr\right)\sin\ph\,,
\end{aligned}
\;\right\}\;
0 \le \ph \le 2\pi\,,
\eeqn
where $k = 1$ for $\Cp$, and $k = -1$ for $\Cm$.

Now we are looking for the singuarities (cusps) of $\Cm$ as may be seen in Figures \ref{Fig:Intersections1}, \ref{Fig:Intersections2} and \ref{Fig:Support_function}. The derivatives of the parametric representation of $\Cm$ are
\beq
\begin{aligned}
  x' 
= {} & p'\cos\ph - (p-r)\sin\ph - p''\sin\ph - p'\cos\ph
= -(p - r + p'')\sin\ph\,,\\
  y' 
= {} & p'\sin\ph + (p-r)\cos\ph + p''\cos\ph - p'\sin\ph
= (p - r + p'')\cos\ph\,. 
\end{aligned}
\eeq
So in order to find the singularities of $\Cm$ we have to solve the equation
\beq
  p(\ph) - r + p''(\ph) = 0
\eeq
for $\ph$. One finds the solutions
\beq
  \ph
= \pm\frac{1}{2}\arccos\frac{2(a^2b^2/r)^{2/3}-a^2-b^2}{a^2-b^2}\,.
\eeq
These solutions are real if
\beq
  -1 \le \frac{2(a^2b^2/r)^{2/3}-a^2-b^2}{a^2-b^2} \le 1\,.
\eeq
From the left and the right inequality it follows $r \le a^2/b$ and $r \ge b^2/a$, respectively.
This means that singularities of $\Cm$ occur only if $b^2/a \le r \le a^2/b$, where $b^2/a$ and $a^2/b$ are, respectively, the minimum and maximum radius of curvature of the ellipse $C_1$.
The solutions in the interval $0\le\ph\le 2\pi$ are given by
\beq
  \ph = \lambda\,,\quad 
  \ph = \pi-\lambda\,,\quad
  \ph = \pi+\lambda\,,\quad
  \ph = 2\pi-\lambda
\eeq
with
\beq
  \lambda
= \frac{1}{2}\,\arccos\frac{2(a^2b^2/r)^{2/3}-a^2-b^2}{a^2-b^2}
  \quad\mbox{if}\quad
  \frac{b^2}{a} \le r \le \frac{a^2}{b}\,.
\eeq

One gets the parametric presentation of the evolute $C_1^\tn{e}$ of $C_1$:
\beq
  x
= \frac{a^2\cos\ph}{p(\ph)}\left(1-\frac{b^2}{p^2(\ph)}\right),\qquad
  y
= \frac{b^2\sin\ph}{p(\ph)}\left(1-\frac{a^2}{p^2(\ph)}\right).
\eeq
Now we set the parameter functions of the evolute equal to that of $\Cm$. This yields
\beq
  \frac{a^2}{p} - r = \frac{a^2}{p} - \frac{a^2b^2}{p^3}
  \;\,\Longrightarrow\;\, \frac{a^2b^2}{p^3} = r\,,\qquad
  \frac{b^2}{p} - r = \frac{b^2}{p} - \frac{b^2a^2}{p^3}
  \;\,\Longrightarrow\;\, \frac{a^2b^2}{p^3} = r\,.
\eeq
We see that $\Cm$ and $C_1^\tn{e}$ have common points. Solving the equation
\beq
  a^2b^2 = rp^3(\ph)
\eeq
for $\ph$, one finds that the singularities of $\Cm$ are these common points.

\begin{figure}[h]
\begin{minipage}[c]{0.32\textwidth}
	\centering
	\includegraphics[scale=1.6]{
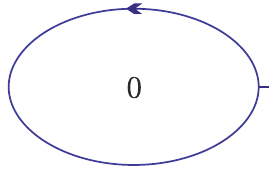}\\
	a) $0 < r < b^2/a$
\end{minipage}
\hfill
\begin{minipage}[c]{0.32\textwidth}
	\centering
	\vspace{0.1cm}
	\includegraphics[scale=1.6]{
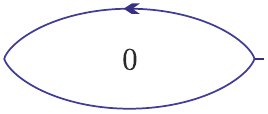}\\[0.1cm]
	b) $r = b^2/a$	
\end{minipage}
\hfill
\begin{minipage}[c]{0.32\textwidth}
	\centering
	\vspace{0.1cm}
	\includegraphics[scale=1.6]{
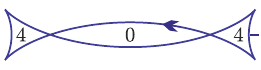}\\
	c) $b^2/a < r < b$
\end{minipage}\\[-0.8cm]
\makebox[1\textwidth][c]{
\begin{minipage}[c]{0.27\textwidth}
	\centering
	\vspace{0.6cm}
	\includegraphics[scale=1.6]{
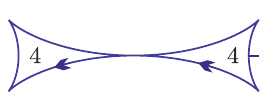}\\[-0.3cm]
	d) $r = b$
\end{minipage}
\hspace{1cm}
\begin{minipage}[c]{0.27\textwidth}
	\centering
	\vspace{0.5cm}
	\includegraphics[scale=1.6]{
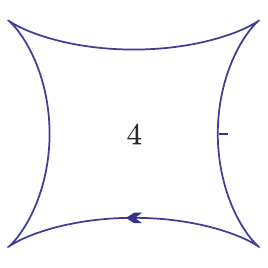}\\[-0.5cm]
	e) $b < r < a$	
\end{minipage}
}\\[0.4cm]
\begin{minipage}[c]{0.24\textwidth}
	\centering
	\includegraphics[scale=1.6]{
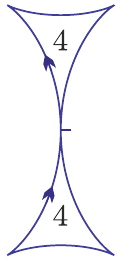}\\[0.1cm]
	f) $r = a$
\end{minipage}
\hfill
\begin{minipage}[c]{0.24\textwidth}
	\centering
	\includegraphics[scale=1.6]{
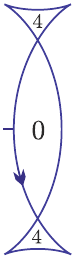}\\[0.1cm]
	g) $a < r < a^2/b$
\end{minipage}
\hfill
\begin{minipage}[c]{0.24\textwidth}
	\centering
	\includegraphics[scale=1.6]{
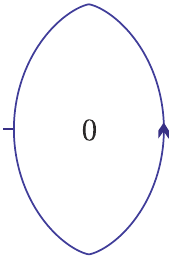}\\[0.1cm]
	h) $r = a^2/b$
\end{minipage}
\hfill
\begin{minipage}[c]{0.24\textwidth}
	\centering
	\includegraphics[scale=1.6]{
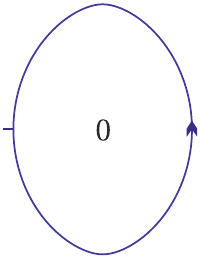}\\[0.1cm]
	i) $a^2/b < r < \infty$
\end{minipage}
\caption{Types of parallel curves $\Cm$ of one ellipse $C_1$ (in different scales)}
\label{Fig:Types}
\end{figure}

Fig.\ \ref{Fig:Types} shows all possible types of inner parallel curves $\Cm$ of $C_1=C_1^*$ for fixed values of $a$ and $b$. We denote by $\#(C_0\cap C_1)$ the number of intersection points of $C_0$ and $C_1$ counted with its multiplicities. If $M_0$ lies in the open set bounded by $\Cm$ and $\Cp$, then there are two distinct intersection points, hence $\#(C_0\cap C_1) = 2$. If $M_0\in\Cp$, then $C_1$ touches $C_0$ in one point with multiplicity $2$, hence also $\#(C_0\cap C_1) = 2$. We give some comments to the shown cases:
\begin{itemize}[leftmargin=0.6cm] \setlength{\itemsep}{-1pt}
\item[a)] Since the smallest radius of curvature of $C_1$ is equal to $b^2/a$, it follows that $C_0$ and $C_1$ cannot intersect in four points. If $M_0$ lies in the open set bounded by $\Cm$, then $\#(C_0\cap C_1) = 0$. If $M_0\in\Cm$, then $C_0$ touches $C_1$ in one point with multiplicity two, hence $\#(C_0\cap C_1) = 2$.
\item[b)] The situation is the same as in Case (a) if $M_0$ does not coincide with a cusp. If $M_0$ lies in one of the two cusps, then $C_0$ touches $C_1$ in one of the two points with smallest radius of curvature, and $C_0$ is the osculating circle in such a point (multiplicity four $\Rightarrow\#(C_0\cap C_1) = 4$). 
\item[c)] We have $\#(C_0\cap C_1) = 0$ if $M_0$ lies in the open set bounded by the inner loop of $\Cm$, and four distinct intersection points, hence $\#(C_0\cap C_1) = 4$, if $M_0$ lies in one of the open sets bounded by the outer loops of $\Cm$. If $M_0$ lies on the inner loop without the double points, then $C_0$ touches $C_1$ in one point with multiplicity two ($\#(C_0\cap C_1) = 2$). If $M_0$ coincides with one double point, then $C_0$ touches $C_1$ in two points, each with multiplicity two, hence $\#(C_0\cap C_1) = 4$. If $M_0$ lies on an outer loop without the cusps (and the double point), then $C_0$ touches $C_1$ in one point and there are two distinct intersection points, hence $\#(C_0\cap C_1) = 4$. If $M_0$ coincides with one of the cusps, then $C_0$ is the osculating circle in the touching point (multiplicity three) and there is a distinct intersection point (multiplicity one), hence $\#(C_0\cap C_1) = 4$.
\item[d)] There are four distinct intersection points if $M_0$ lies in the open set bounded by $\Cm$. If $M_0$ lies on $\Cm$ but not in the cusps and the self touching point, then $C_0$ touches $C_1$ in one point with multiplicity two and there are two distinct intersection points. If $M_0$ coincides with one cusp, then $C_0$ touches $C_1$ in one point where $C_0$ is the osculating circle, and there is one distinct intersection point. If $M_0$ coincides with the self touching point, then $C_0$ touches $C_1$ in two points, each with multiplicity two. In all of these subcases we have $\#(C_0\cap C_1) = 4$.
\item[e)] The situation is the same as in Case (d) with the exception that there is no self touching point. 
\item[f)] See Case (d).
\item[g)] See Case (c). 
\item[h)] The situation is the same as in Case (i) if $M_0$ does not coincide with a cusp. If $M_0$ lies in one of the two cusps, then $C_0$ touches $C_1$ in one of the two points with largest radius of curvature, and $C_0$ is the osculating circle in such a point (multiplicity four $\Rightarrow\#(C_0\cap C_1) = 4$).
\item[i)] Since the largest radius of curvature of $C_1$ is equal to $a^2/b$, it follows that $C_0$ and $C_1$ cannot intersect in four points. If $M_0$ lies in the open set bounded by $\Cm$, then $\#(C_0\cap C_1) = 0$. If $M_0\in\Cm$, then $C_0$ touches $C_1$ in one point with multiplicity two, hence $\#(C_0\cap C_1) = 2$.
\end{itemize}

\noindent
The orientations of $\Cm$ shown in Fig.\ \ref{Fig:Types} are the orientations resulting from \eqref{Eq:Param_C1k}. The starting points with value $\phi=0$ are marked with small line segments.

\section{Areas and measures for ellipses with fixed direction}
\label{Sec:Measures1}

We consider the moving ellipse $C_1^*$ (with fixed direction) and center point $M_1$, and derive expressions for the following areas of sets of positions of $M_1$ 
\beqn \label{Eq:A1}
  \Aoi
:= A(\{M_1\colon C_0\subset K_1^*\})\,,\quad
  \Aio
:= A(\{M_1\colon C_1^*\subset K_0\})
\eeqn
and
\beqn \label{Eq:A2}
  A_{2j}
:= A(\{M_1\colon \#(C_0\,\cap\,C_1^*) = 2j\})\,,\quad j\in\{1,2\}\,.
\eeqn
We denote by $m(\{C_1^*\colon X\})$ the measure of the set of positions of $C_1^*$ with property $X$. Since $C_1^*$ has fixed direction,
\begin{gather*}
  m(\{C_1^*\colon C_0\subset K_1^*\}) = \Aoi\,,\quad 
  m(\{C_1^*\colon C_1^*\subset K_0\}) = \Aio\,,\\[0.1cm]
  m(\{C_1^*\colon \#(C_0\,\cap\,C_1^*) = 2j\} = A_{2j}\,.
\end{gather*}
In the following, $E(\ph,\eps)$ denotes the incomplete elliptic integral of the second kind,
\beq
  E(\ph,\eps)
= \int_0^\ph\sqrt{1-\eps^2\sin^2\theta}\;\dd\theta\,,
\eeq
and $E(\eps) := E(\eps,\pi/2)$ the complete elliptic integral of the second kind.

\begin{lem} \label{Lem:Lemma}
Depending on the following cases, the areas \eqref{Eq:A1} and \eqref{Eq:A2} are given by

\begin{center}
\renewcommand{\arraystretch}{1.2}
\begin{tabular}{|c|c|c|c|c|c|} \hline
\rule{0pt}{13pt}
\hspace{-0.3cm} Case & Interval & $\Aoi$ & $\Aio$ & $A_2$ & $A_4$\\[2pt] \hline\hline
\rule{0pt}{20pt}
\hspace{-0.25cm} 1 &
$0 < r \le \dfrac{b^2}{a}$ & 
$A^*$ & 
$0$ & 
$8raE(\eps)$ & 
$0$\\[2pt]
\rule{0pt}{18pt}
\hspace{-0.25cm} 2 &
$\dfrac{b^2}{a} < r < b$ & 
$\Ft(\alpha)$ & 
$0$ & 
$2\pi r^2 + 2\pi ab - 2\Ft(\alpha)$ & 
$\Ft(\alpha) - A^*$\\[8pt]
\rule{0pt}{13pt}
\hspace{-0.25cm} 3 &
$b \le r \le a$ &
$0$ & 
$0$ & 
$2\pi r^2 + 2\pi ab$ & 
$-A^*$\\[1pt]
\rule{0pt}{20pt}
\hspace{-0.25cm} 4 &
$a < r < \dfrac{a^2}{b}$ & $0$ & $F(\beta)$ & $2\pi r^2 + 2\pi ab - 2F(\beta)$ &
$F(\beta) - A^*$\\[4pt]
\rule{0pt}{18pt}
\hspace{-0.25cm} 5 & 
$\dfrac{a^2}{b} \le r < \infty$ &
$0$ &
$A^*$ &
$8raE(\eps)$ & 
$0$ \\[6pt] \hline
\end{tabular}
\renewcommand{\arraystretch}{1}
\end{center}

\noindent where
\beq
  F(\ph)
= 2r^2\ph
+ 2ab\arctan\left(\frac{b}{a}\,\tan\ph\right)
- 4raE(\ph,\eps)
+ \frac{ra\eps^2\sin 2\ph}{\sqrt{1-\eps^2\sin^2\ph}}\,,
\eeq
\beq
  A^* 
= F(\pi/2) = \pi r^2 + \pi ab - 4raE(\eps)\,,\quad
  \Ft(\ph)
= A^* - F(\ph)\,,
\eeq
and
\beq
  \alpha
= \arctan\frac{\sqrt{r^2a^2-b^4}}{b\,\sqrt{b^2-r^2}}
  \,,\quad
  \beta
= \arctan\frac{a\,\sqrt{r^2-a^2}}{\sqrt{a^4-r^2b^2}}\,. 
\eeq
\end{lem}

\begin{proof}
We consider the inverse translation with fixed ellipse $C_1 = C_1^*$ and moving circle $C_0$ (of radius $r$). From Section \ref{Sec:Intersections} (see Figures \ref{Fig:Intersections1} and \ref{Fig:Intersections2}) it follows that the areas \eqref{Eq:A1} and \eqref{Eq:A2} are also given by
\begin{gather*}
  \Aoi
= A(\{M_0\colon C_0\subset K_1\})\,,\;\;
  \Aio
= A(\{M_0\colon C_1\subset K_0\})\,,\;\;
  A_{2j}
= A(\{M_0\colon \#(C_0\,\cap\,C_1) = 2j\})\,.
\end{gather*}
Intersection (touching) points with muliplicity $>1$ only occur if $M_0\in\Cp$ and $M_0\in\Cm$, so the sets of positions of $M_0$ with intersection multiplicity $>1$ always have area zero. Therefore, it suffices to consider only positions of $M_0$ with distinct intersection points.

The area $A^+$ of the set enclosed by the outer parallel curve $\Cp$ of $C_1$ in the distance $r$ is given by the integral
\beq
  A^+
= \frac{1}{2}\int_0^{2\pi}[p(\ph)+r]\,[p(\ph)+r+p''(\ph)]\,\dd\ph
\eeq
(see p.\ 3, Eq.\ (1.7) and p.\ 7 in \cite{Santalo}). The area and the perimeter of $K_1$ are given by $\pi ab$ and $4aE(\eps)$, respectively, where
\beqn \label{Eq:Eccentricity}
  \eps
= \frac{\sqrt{a^2-b^2}}{a}
\eeqn
is the eccentricity of $C_1$ (see e.\ g.\ \cite[pp.\ 230-232]{Bosch}). From Eq.\ (1.18) on  p.\ 8 in \cite{Santalo} it follows that
\beqn \label{Eq:A^+}
  A^+
= \pi r^2 + \pi ab + 4raE(\eps)\,.
\eeqn
(This result also follows from Eq.\ (14.5) or (14.6) on p.\ 600 in \cite{SchneiderWeil}.) The (signed) area of the set enclosed by the inner parallel curve $\Cm$ in the distance $r$ is given by the integral
\beqn \label{Eq:A^-}
  A^-
= \frac{1}{2}\int_0^{2\pi}[p(\ph)-r]\,[p(\ph)-r+p''(\ph)]\,\dd\ph
\eeqn
if $\Cm$ has no self intersections, even in the cases in which $\Cm$ is not convex (cp.\ \cite[p.\ 8]{Santalo}). The sign depends on the orientation of the curve. (Only if $r \le b^2/a$, $A_r$ is the area of the {\em interior parallel set}.) For a curve $\Cm$ with self intersections, \eqref{Eq:A^-} gives the sum of the signed areas of its loops depending on the orientation of each loop (see Figures \ref{Fig:Support_function} and \ref{Fig:Types}). Therefore, instead of \eqref{Eq:A^-}, in the following we use
\beq 
  \widetilde{A}^-
= 2\int_{\phi_1}^{\phi_2}[p(\ph)-r]\,[p(\ph)-r+p''(\ph)]\,\dd\ph\,,\quad
  0 \le \ph_1 < \ph_2 \le \frac{\pi}{2}\,,
\eeq
with suitable limits $\ph_1$, $\ph_2$ in order to derive areas $\widetilde{A}^-$ of sets enclosed by loops of $\Cm$, where the factor $2$ results from the symmetry of $\Cm$. 
    
So for the function
\beq
  f(\ph)
= 2\,[p(\ph)-r]\,[p(\ph)-r+p''(\ph)]
\eeq
we have to determine one of its antiderivatives $F(\ph)$,
\begin{align} \label{Eq:F-1}
\begin{split}
 F(\ph)
= {} & 2\int[p(\ph)-r]\,[p(\ph)-r+p''(\ph)]\,\dd\ph\\ 
= {} & 2\int[p(\ph)-r]^2\,\dd\ph
+ 2\int[p(\ph)-r]\,p''(\ph)\,\dd\ph\,.
\end{split}
\end{align}
Since we want to determine one antiverivative, we omit the constant of integation. Using integration by parts in the last integral with $u=p-r$, $u'=p'$, $v'=p''$, $v=p'$, one gets
\begin{align*}
  F(\ph)
= {} & 2\int[p(\ph)-r]^2\,\dd\ph + 2[p(\ph)-r]\,p'(\ph) 
- 2\int p'^2(\ph)\,\dd\ph\displaybreak[0]\\
= {} & 2r^2\ph + 2\int\left[p^2(\ph)-p'^2(\ph)\right]\dd\ph
- 4r\int p(\ph)\,\dd\ph + 2\,[p(\ph)-r]\,p'(\ph)\,.
\end{align*}
After the rearrangement
\begin{align*}
  p^2(\ph)-p'^2(\ph)
= {} & p^2(\ph) - \frac{(a^2-b^2)^2\cos^2\ph\sin^2\ph}{p^2(\ph)}
	\displaybreak[0]\\
= {} & p^2(\ph) - a^2\sin^2\ph - b^2\cos^2\ph 
- \frac{(a^2-b^2)^2\cos^2\ph\sin^2\ph}{p^2(\ph)}
+ a^2\sin^2\ph + b^2\cos^2\ph\displaybreak[0]\\
= {} & (a^2-b^2)(\cos^2\ph-\sin^2\ph) + \frac{a^2b^2}{p^2(\ph)}
= (a^2-b^2)\cos 2\ph + \frac{a^2b^2}{p^2(\ph)} 
\end{align*}
we find
\beqn \label{Eq:Integral1}
  \int\left[p^2(\ph)-p'^2(\ph)\right]\dd\ph
= \frac{(a^2-b^2)\sin 2\ph}{2} 
+ a^2b^2\int\frac{\dd\ph}{p^2(\ph)}\,.
\eeqn
Now we consider the last integral
\beq
  \int\frac{\dd\ph}{p^2(\ph)}
= \int\frac{\dd\ph}{a^2\cos^2\ph + b^2\sin^2\ph}
\eeq
which may be written as
\beq
\frac{1}{a^2}\int\frac{1}{1 + (b/a)^2\tan^2\ph}\,
  \frac{\dd\ph}{\cos^2\ph}
\eeq
(see \cite[p.\ 115, Eq.\ (235)]{Chemnitius}). The substitution
\beq
  z = \frac{b}{a}\,\tan\ph\,,\qquad
  \frac{\dd z}{\dd\ph} = \frac{b}{a\cos^2\ph}\,,\qquad
  \frac{\dd\ph}{\cos^2\ph} = \frac{a}{b}\,\dd z 
\eeq
gives
\begin{align} \label{Eq:Chemnitius235}
  \int\frac{\dd\ph}{p^2(\ph)}
= {} & \frac{1}{ab}\int\frac{\dd z}{1+z^2}
= \frac{1}{ab}\,\arctan z
= \frac{1}{ab}\,\arctan\left(\frac{b}{a}\,\tan\ph\right).
\end{align}
The support function \eqref{Eq:support1} may be written as
\beq 
  p(\ph)
= a\,\sqrt{1-\eps^2\sin^2\ph}
\eeq
with $\eps$ according to \eqref{Eq:Eccentricity}.
Now one gets
\beqn \label{Eq:Elliptic_integral}
  \int p(\ph)\,\dd\ph
= a\int\sqrt{1-\eps^2\sin^2\ph}\;\dd\ph 
= aE(\ph,\eps)\,,
\eeqn
and 
\begin{align} \label{Eq:Term}
  [p(\ph)-r]\,p'(\ph)
= \frac{ra\eps^2\sin 2\ph}{2\sqrt{1-\eps^2\sin^2\ph}}
- \frac{1}{2}\,a^2\eps^2\sin 2\ph\,.
\end{align}
Taking \eqref{Eq:Integral1}, \eqref{Eq:Chemnitius235}, \eqref{Eq:Elliptic_integral} and \eqref{Eq:Term} into account, an antiderivative $F(\ph)$ (see \eqref{Eq:F-1}) is given by
\begin{align*}
  F(\ph)
= {} & 2r^2\ph + (a^2-b^2)\sin 2\ph
+ 2ab\arctan\left(\frac{b}{a}\,\tan\ph\right)
- 4raE(\alpha,\eps)\\ 
& + \frac{ra\eps^2\sin 2\ph}{\sqrt{1-\eps^2\sin^2\ph}}
 - a^2\eps^2\sin 2\ph\,,
\end{align*}
hence
\beq 
  F(\ph)
= 2r^2\ph 
+ 2ab\arctan\left(\frac{b}{a}\,\tan\ph\right)
- 4raE(\alpha,\eps)  
+ \frac{ra\eps^2\sin 2\ph}{\sqrt{1-\eps^2\sin^2\ph}}\,.
\eeq

\newpage\noindent 
Now we are able to investigate the five cases according to Lemma \ref{Lem:Lemma}. For this purpose we use Fig.~\ref{Fig:Types}.
\begin{itemize}[leftmargin=0.6cm]
\item[1)] (Figures \ref{Fig:Types} a and b) We have $A_4 = 0$. $\Cm$ is positively oriented, hence
\beq
  \Aoi
= F(\pi/2) - F(0)
= F(\pi/2)
= \pi r^2 + \pi ab - 4raE(\eps)
=: A^*\,.
\eeq
Since $b > r$, $C_1$ cannot be contained in $K_0$, hence $\Aio = 0$. With $A^+$ according to \eqref{Eq:A^+}, we get
\beq
  A_2
= A^+ - \Aoi
= \pi r^2 + \pi ab + 4raE(\eps) 
- \left(\pi r^2 + \pi ab - 4raE(\eps)\right)
= 8raE(\eps)\,.
\eeq
\item[2)] (Fig.\ \ref{Fig:Types} c) Here, as in Case 1, $\Aio = 0$. We denote by $\alpha$ the first value of $\ph$ belonging to a self intersection point. Here the $y$-coordinate of $\Cm$ is equal to zero. From \eqref{Eq:Param_C1k} we get
\beq
  \frac{b^2}{p(\alpha)} - r = 0
\eeq
which yields
\beq
  \alpha
= \arctan\frac{\sqrt{r^2a^2-b^4}}{b\,\sqrt{b^2-r^2}}\,.
\eeq
The middle loop of $\Cm$ is positively oriented, hence
\beq
  \Aoi
= F(\pi/2) - F(\alpha)
= A^* - F(\alpha)
= \Ft(\alpha)\,.
\eeq
The outer loops of $\Cm$ are negatively oriented, hence
\beq
  A_4
= -(F(\alpha) - F(0))
= -F(\alpha)
= -\left(A^*-\Ft(\alpha)\right)
= \Ft(\alpha) - A^*\,. 
\eeq
It follows that
\begin{align*}
 A_2
= {} & A^+ - \Aoi - A_4\\
= {} & \pi r^2 + \pi ab + 4raE(\eps) - \Ft(\alpha)
- \left[\Ft(\alpha) - (\pi r^2 + \pi ab - 4raE(\eps))\right]\\
= {} & 2\pi r^2 + 2\pi ab - 2\Ft(\alpha)\,.
\end{align*}
\item[3)] (Fig.\ \ref{Fig:Types} d, e, f) Here one easily sees that $\Aoi = 0 = \Aio$. $\Cm$ is negatively oriented, hence
\beq
  A_4
= -(F(\pi/2) - F(0))
= -F(\pi/2)
= -A^*. 
\eeq
It follows that
\begin{align*}
  A_2
= {} & A^+ - A_4
= \pi r^2 + \pi ab + 4raE(\eps)
- \left(4raE(\eps) - \pi r^2 - \pi ab\right)
= 2\pi r^2 + 2\pi ab\,.
\end{align*}
\item[4)] (Fig.\ \ref{Fig:Types} g) Since $r > a$, $C_0$ cannot be contained in $K_1$, hence $\Aoi = 0$. We denote by $\beta$ the first value of $\ph$ belonging to a self intersection point. Here the $x$-coordinate of $\Cm$ is equal to zero. From \eqref{Eq:Param_C1k} we get 
\beq
  \frac{a^2}{p(\beta)} - r = 0
\eeq
which yields
\beq
  \beta
= \arctan\frac{a\,\sqrt{r^2-a^2}}{\sqrt{a^4-r^2b^2}}\,. 
\eeq
The middle loop of $\Cm$ is positively oriented, hence
\beq
  \Aio
= F(\beta) - F(0)
= F(\beta)
\eeq
The outer loops of $\Cm$ are negatively oriented, hence
\beq
  A_4
= -\left(F(\pi/2) - F(\beta)\right)
= F(\beta) - A^*, 
\eeq
and
\begin{align*}
  A_2
= {} & A^+ - \Aio - A_4\\
= {} & \pi r^2 + \pi ab + 4raE(\eps) - F(\beta) 
- \left[F(\beta) - (\pi r^2 + \pi ab - 4raE(\eps))\right]\\
= {} & 2\pi r^2 + 2\pi ab - 2F(\beta)\,.  
\end{align*}
\item[5)] (Fig.\ \ref{Fig:Types} i) We have $A_4 = 0$. $\Cm$ is positively oriented, hence
\beq
  \Aio
= F(\pi/2) - F(0)
= A^*.
\eeq
Finally, we get
\beq
  A_2
= A^+ - \Aio
= \pi r^2 + \pi ab + 4raE(\eps) 
- \left(\pi r^2 + \pi ab - 4raE(\eps)\right)
= 8raE(\eps)\,. \qedhere
\eeq
\end{itemize}
\end{proof}

\section{Measures for oriented ellipses}
\label{Sec:Measures2}

Now let us go back to the original motion of the ellipse $C_1$ with respect to the fixed circle $C_0$. We give up the assumption that the direction of $C_1$ is fixed. In the following we consider $C_1$ as oriented. For this orientation we attach a frame $\xi,\eta$ to $C_1$ with its origin in the center point $M_1$ of $C_1$ (see Fig.\ \ref{Fig:Motion01}). Furthermore, we define this measures
\beq
  \mio := m(\{C_1\colon C_1\subset K_0\})\,,\quad
  \moi := m(\{C_1\colon C_0\subset K_1\})\,, 
\eeq
\beq
  m_2 := m(\{C_1\colon \#(C_0 \cap C_1) = 2\})\,,\quad
  m_4 := m(\{C_1\colon \#(C_0 \cap C_1) = 4\})\,. 
\eeq
for $C_1$.

\begin{thm} \label{Thm:Theorem}
The measure for all oriented ellipses with semi-major axis of length $a$ and semi-minor axis of length $b$ which intersect a fixed circle of radius $r$ in two and four points are given by
\begin{align} \label{Eq:m2-m4}
  m_2
= {} & 4\pi^2r^2 + 4\pi^2ab - 2\mI
  \quad\mbox{and}\quad
  m_4
= 8\pi raE(\eps) - 2\pi^2r^2 - 2\pi^2ab + \mI\,, 
\end{align}
respectively, where, with the cases of Lemma \ref{Lem:Lemma},
\beqn \label{Eq:mi}
  \mI = \left\{
  \begin{array}{l@{\quad\mbox{if}\quad}r@{\:\,}c@{\:\,}ll}
	\mio = 2\pi\Aio & a & < & r & \mbox{(Cases 4 and 5)}\,,\\
	0    & (a & \ge & r) \land (b \le r) & \mbox{(Case 3)}\,,\\
	\moi = 2\pi\Aoi & b & > & r & \mbox{(Cases 1 and 2)}\,,
  \end{array}
  \right\}
\eeqn
and
\beq
  \Aio = \left\{
  \begin{array}{l@{\quad\mbox{if}\quad}ll}
	\pi r^2+\pi ab-4raE(\eps) & a^2/b\le r & \mbox{(Case 5)}\,,\\[0.1cm]
	F(\beta) & a^2/b>r & \mbox{(Case 4)}\,,
  \end{array}
  \right.
\eeq
\beq
  \Aoi = \left\{
  \begin{array}{l@{\quad\mbox{if}\quad}ll}
	\Ft(\alpha) & b^2/a<r	& \mbox{(Case 2)}\,,\\[0.1cm]
	\pi r^2+\pi ab-4raE(\eps) & b^2/a\ge r & \mbox{(Case 1)}\,.  
  \end{array}
  \right.
\eeq
\end{thm}

\begin{figure}[h]
  \begin{center}
	\includegraphics[scale=1]{
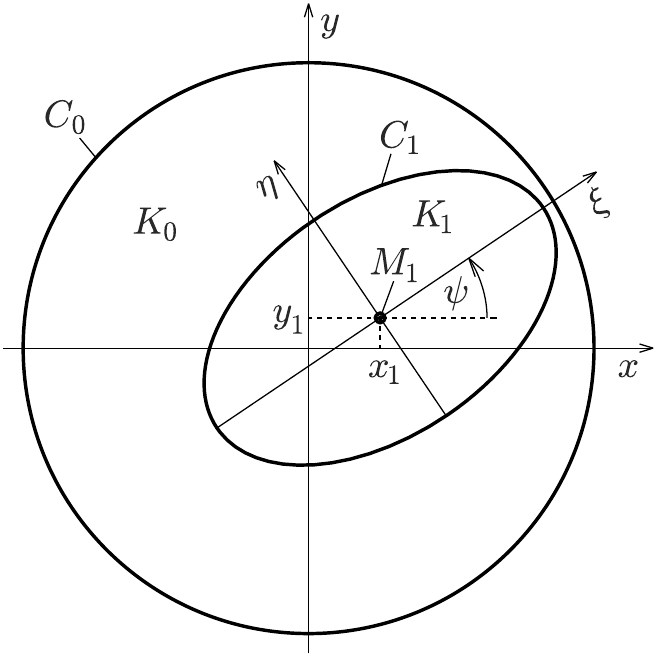}
  \end{center}
  \vspace{-0.6cm}
  \caption{\label{Fig:Motion01} For the proof of Theorem \ref{Thm:Theorem}}
\end{figure}

\begin{proof}
We have
\beq
  \mio
= m(\{C_1\colon C_1\subset K_0\})
= m(\{K_1\colon K_1\subset K_0\})
= \int_{\{K_1\,:\,K_1\subset K_0\}}\dd K_1\,,
\eeq
where $\dd K_1$ denotes the kinematic density of $K_1$ (see \cite[pp.\ 85-89]{Santalo}). We can write it as $\dd K_1 = \dd x_1\wedge\dd y_1\wedge\dd\psi$, where $x_1$, $y_1$ are the coordinates of the center point $M_1$ of $K_1$ with respect to the fixed $x,y$-frame and $\psi$ is the angle between the $x$-axis of the fixed frame and the $\xi$-axis of the moving frame (see Fig.\ \ref{Fig:Motion01}). We get
\beq
  \mio
= 2\pi\int\dd x_1\wedge\dd y_1\,,  
\eeq
where the integral has to be taken over the points $M_1$ such that $K_1\subset K_0$ for fixed angle $\psi$, hence
\begin{align*}
  \mio
= {} & 2\pi A(\{M_1\colon K_1^*\subset K_0\})
= 2\pi A(\{M_1\colon C_1^*\subset K_0\})\\[0.1cm]
= {} & 2\pi A(\{M_0\colon C_1^*\subset K_0\})
= 2\pi\Aio\,,
\end{align*}
where
\beq
  \Aio = \left\{
  \begin{array}{ll}
	\pi r^2+\pi ab-4raE(\eps) & \mbox{if $a^2/b\le r$}\quad 
		\mbox{(Case 5 in Lemma \ref{Lem:Lemma})}\,,\\[0.1cm]
	F(\beta) & \mbox{if $a^2/b>r$}\quad 
		\mbox{(Case 4)}\,,\\[0.1cm]
	0 & \mbox{in Cases 1-3}\,.
  \end{array}
  \right.
\eeq
Analogously, one finds $\moi = 2\pi\Aoi$, $m_2 = 2\pi A_2$ and $m_4 = 2\pi A_4$, with $\Aoi$, $A_2$, $A_4$ from Lemma \ref{Lem:Lemma}. 	
\end{proof}

\section{Hitting probabilities}
\label{Sec:Probabilities}

Now we consider the random throw of an ellipse $C_1$ onto an unbounded lattice of circles $C_0$ of radius $r$ as it is shown in Fig.\ \ref{Fig:Reseau}. The center points $M_0$ of the circles $C_0$ lie on the vertices of parallelograms with sides of length $s$ and $t$, and angle $\sigma$, $0<\sigma\le\pi/2$. We assume that $C_1$ can hit at most one $C_0$ at the same time. Due to the periodicity of the lattice it suffices to consider only one parallelogram $\mathcal{P}$ for the calculation of the hitting probabilities, for which we choose
\beq
  \mathcal{P}
= \{(x,y)\in\mathbb{R}^2\colon 0\le y\le t\sin\sigma\,,\;
	y\cot\sigma\le x\le s+y\cot\sigma\}\,.
\eeq
Now, we define the {\em random throw} as follows: The coordinates $x_1$, $y_1$ of $M_1$ are random variables uniformly distributed in $[y\cot\sigma, s+y\cot\sigma]$ and $[0, t\sin\sigma]$, respectively; the angle $\psi$ between the $x$-axis of the fixed frame and the $\xi$-axis of the frame attached to $C_1$ is uniformly distributed in $[0,2\pi]$. All three random variables are stochastically independent.

The total measure for all positions of the oriented ellipse $C_1$ with center point $M_1$ in $\mathcal{P}$ is given by
\beq
  m_\tn{t}
= \int_{\{C_1\,:\,M_1\in\,\mathcal{P}\}}\dd x_1\wedge\dd y_1\wedge\dd\psi
= 2\pi\int_{\{M_1\,:\,M_1\in\,\mathcal{P}\}}\dd x_1\wedge\dd y_1
= 2\pi st\sin\sigma\,.
\eeq

The events $C_1\subset K_0$ and $C_0\subset K_1$ are mutually exclusive,
\beq
  (C_1\subset K_0) \land (C_0\subset K_1) = \emptyset\,.
\eeq
So we may write the measure $m_\tn{i}$ (see \eqref{Eq:mi}) as
\beq
  \mI
= m(\{C_1 \colon (C_1\subset K_0) \lor (C_0\subset K_1) 
	\ne \emptyset\})\,.
\eeq
We define the measure
\beq
  m_\tn{e}
= m(\{C_1\colon K_0\cap K_1 = \emptyset\})
= m(\{K_1\colon K_0\cap K_1 = \emptyset\})
\eeq
and find
\begin{align*}
  m_\tn{e}
= {} & m_\tn{t} - (m_2 + m_4 + \mI)
= m_\tn{t} - m(\{K_1\colon K_0\cap K_1 \ne \emptyset \})\\
= {} & 2\pi st\sin\sigma
- \left[2\pi^2r^2 + 2\pi^2 ab + 8\pi raE(\eps)\right].
\end{align*}
 
Now we define the hitting probabilities
\beq
  p_\tn{i}
= P((C_1\subset K_0) \lor (C_0\subset K_1) \ne \emptyset)\,,\qquad
  p_\tn{e}
= P(K_0 \cap K_1 = \emptyset)
= 1 - P(K_0 \cap K_1 \ne \emptyset) 
\eeq
and, for the number of intersection points,
\beq
  p_{2j}
= P(\#(C_0\cap C_1) = 2j)\,,\quad j\in\{0,1,2\}\,.
\eeq
So with
\beq
  p_2 = \frac{m_2}{m_\tn{t}}\,,\quad
  p_4 = \frac{m_4}{m_\tn{t}}\,,\quad
  p_\tn{i} = \frac{\mI}{m_\tn{t}}\,,\quad
  p_\tn{e} = \frac{m_\tn{e}}{m_\tn{t}}
\eeq
and
\beq
  p_0 = 1 - p_2 - p_4 = p_\tn{i} + p_\tn{e}
\eeq
we get the following corollary.  	  

\begin{cor} \label{Cor:Corollary}
Under the assumption $2(a+r) \le \min(s,t)$,
\begin{gather*}
  p_0
= 1 - \frac{2\pi^2r^2 + 2\pi^2ab + 8\pi raE(\eps) - \mI}
	{2\pi st\sin\sigma}\,,\quad
  p_2
= \frac{2\pi^2r^2 + 2\pi^2ab - \mI}{\pi st\sin\sigma}\,,\\[0.2cm]
  p_4 
= \frac{8\pi raE(\eps) - 2\pi^2r^2 - 2\pi^2 ab + \mI}
	{2\pi st\sin\sigma}\,,\;\;\,
 \pI
= \frac{\mI}{2\pi st\sin\sigma}\,,\;\;\,
 \pe
= 1 - \frac{2\pi^2r^2 + 2\pi^2ab + 8\pi raE(\eps)}{2\pi st\sin\sigma}
\end{gather*}
with $\mI$ according to Theorem \ref{Thm:Theorem}.
\end{cor}

\section{Special case: Line segment} \label{Sec:Special_case}

Now we consider for the ellipse $C_1$ the special case $a \ne 0$, $b = 0$ so that it degenerates to a pair of coinciding line segments of length $\ell:=2a$. Since $b = 0 < r$, the Cases 1 and 2 (see Thm \ref{Thm:Theorem}) cannot occur. Since $a^2/b = a^2/0 = \infty > r$, the Case 5 cannot occur, so Case 3 and Case 4 are the remaining cases. We have
\beq
  \eps = 1 \quad\mbox{and}\quad E(1) = 1\,.
\eeq
Therefore, considering the two coinciding line segments as one line segment, from \eqref{Eq:m2-m4} follows
\beqn \label{Eq:m2-m4-Santalo}
  m_1
= 4\pi^2r^2 - 2\mI\,,\quad
  m_2
= 4\pi r\ell - 2\pi^2r^2 + \mI\,. 
\eeqn
In Case 3 we have $\ell \ge 2r$ and $\mI = 0$. Therefore, $m_1 = \mathfrak{M}_1$ and $m_2 = \mathfrak{M}_2$, where $\mathfrak{M}_1$ and $\mathfrak{M}_2$ are Santal\'o's terms for the measures in \eqref{Eq:M1_M2} with $\mathfrak{M}_\tn{i} = 0$. In Case 4 it remains to show that
\beq
  \mI
= 2\pi\AI = 2\pi F(\beta)
= 2\pi\left[2r^2\beta
+ 2ab\arctan\left(\frac{b}{a}\,\tan\beta\right)
- 4raE(\beta,\eps)
+ \frac{ra\eps^2\sin 2\beta}{\sqrt{1-\eps^2\sin^2\beta}}\right]
\eeq
with
\beq
  \beta 
= \arctan\frac{\sqrt{r^2-a^2}}{a}
= \arctan\frac{\sqrt{1-(a/r)^2}}{a/r}
\eeq
is equal to $\mathfrak{M}_\tn{i}$ according to \eqref{Eq:Mi}. The expression for the angle $\beta$ may be written as
\beq
  \beta
= \frac{\pi}{2} - \arcsin\frac{a}{r}\
= \frac{\pi}{2} - \arcsin\frac{\ell}{2r}\,. 
\eeq
(In \cite[Vol.\ 1, p.\ 76]{Gradstein_Ryshik}, Eq. (1.624.3) states incorrectly that
\beq
  \arctan\frac{\sqrt{1-x^2}}{x}
= \arcsin x\,,\quad 0< x \le 1\,.
\eeq
The correct formula, which we use, is
\beq
  \arctan\frac{\sqrt{1-x^2}}{x}
= \frac{\pi}{2} - \arcsin x\,.)
\eeq
Next we have
\beq
  \arctan\left(\frac{b}{a}\,\tan\beta\right)
= \arctan\left(0\,\tan\beta\right)
= 0\,.
\eeq
With
\begin{align*}
  \sin\beta
= {} & \sin\left(\arctan\frac{\sqrt{r^2-a^2}}{a}\right)
= \sqrt{1-\frac{a^2}{r^2}}
\end{align*}
one finds that
\beq
  E(\beta,\eps)
= E(\beta,1)
= E\left(\arctan\frac{\sqrt{r^2-a^2}}{a}\,,\,1\right)
= \sqrt{1-\frac{a^2}{r^2}}
= \frac{1}{r}\,\sqrt{r^2-\frac{\ell^2}{4}}
\eeq
as well as
\beq
  \frac{\eps^2\sin 2\beta}{\sqrt{1-\eps^2\sin^2\beta}}
= \frac{2\eps^2\sin\beta\cos\beta}{\sqrt{1-\eps^2\sin^2\beta}}
= \frac{2\sin\beta\cos\beta}{\sqrt{1-\sin^2\beta}}
= 2\sin\beta
= 2\,\sqrt{1-\frac{a^2}{r^2}}
= \frac{2}{r}\,\sqrt{r^2-\frac{\ell^2}{4}}\,.
\eeq
It follows that
\begin{align*}
  F(\beta)
= {} & 2r^2\left(\frac{\pi}{2} - \arcsin\frac{\ell}{2r}\right) + 0
- 4r\frac{\ell}{2}\frac{1}{r}\,\sqrt{r^2-\frac{\ell^2}{4}}
+ r\frac{\ell}{2}\frac{2}{r}\,\sqrt{r^2-\frac{\ell^2}{4}}\\
= {} & \pi r^2 - 2r^2\arcsin\frac{\ell}{2r}
- \ell\,\sqrt{r^2-\frac{\ell^2}{4}}\,, 
\end{align*}
so $\mI = 2\pi F(\beta)$ is indeed Santal\'o's formula for $\mathfrak{M}_\tn{i}$ (see \eqref{Eq:Mi}, or (32) in \cite[p.\ 165]{Santalo40}) for our Case~4 in which $\ell < 2r$ holds.

The hitting probabilities in Corollary \ref{Cor:Corollary} turn into 
\beqn \label{Eq:Measure-prob-needle}
\left.
\begin{array}{c}
\ds{p_0
= 1 - \frac{2\pi^2 r^2 + 4\pi r\ell - \mI}
	{2\pi st\sin\sigma}\,,\quad
  p_1
= \frac{2\pi^2 r^2 - \mI}{\pi st\sin\sigma}\,,\quad
  p_2 
= \frac{4\pi r\ell - 2\pi^2 r^2 + \mI}{2\pi st\sin\sigma}\,,}\\[0.4cm]
\ds{\pI
= \frac{\mI}{2\pi st\sin\sigma}\,,\quad
  \pe	
= 1 - \frac{2\pi^2 r^2 + 4\pi r\ell}
	{2\pi st\sin\sigma}\,.}
\end{array}
\right\}
\eeqn
These are Santal\'o's probabilities for the line segment under the assumption that the line segment can hit only one circle \cite[p.\ 166, Eq.\ (36)]{Santalo40}.

\section{Some comments}

In all five cases of Theorem \ref{Thm:Theorem} one finds that $2A_2 + 4A_4
= 16raE(\eps)$, hence
\begin{align*}
  2m_2 + 4m_4
= {} & 2\pi(2A_2 + 4A_4)
= 32\pi raE(\eps)
= 4\times 2\pi r\times 4aE(\eps)\\
= {} & 4 \times (\tn{length of $C_0$}) \times (\tn{length of $C_1$})\,.
\end{align*}
This generally follows from Poincar\'e's formula (see Eq. (7.10) in \cite[p.\ 111]{Santalo}). Analogously, from Corollary \ref{Cor:Corollary} it follows that the expected value for the random number $Z := \#(C_0\cap C_1)$ of intersection points is given by
\begin{align*}
  \mathbb{E}(Z) 
= {} & \sum_{j=1}^4 jp_j
= 2p_2 + 4p_4
= \frac{16raE(\eps)}{st\sin\sigma}
= \frac{2\cdot 2\pi r\cdot 4aE(\eps)}{\pi st\sin\sigma}\\
= {} & \frac{2\times(\tn{length of $C_0$}) \times (\tn{length of $C_1$})}
		{\pi st\sin\sigma}\,.
\end{align*}
One also gets $\mathbb{E}(Z)$ from Eq. (8.11) in \cite[p.\ 134]{Santalo}.  

The measure $m(\{K_1\colon K_0\cap K_1\ne\emptyset\})$ also follows from the {\em principal kinematic formula} (see Theorem 5.1.3 in \cite[p.\ 175]{SchneiderWeil}).

Santal\'o only calculated the measure $\mathfrak{M}_\tn{i}$ and derived the measures $\mathfrak{M}_1$ and $\mathfrak{M}_2$ using Poincar\'e's formula in the form
\beq
  \mathfrak{M}_1 + 2\mathfrak{M}_2
= 4\times(\tn{length of $C_0$}) \times (\tn{length of line segment})
= 8\pi r\ell 
\eeq
(see \cite[pp.\ 164-165]{Santalo40}, Equations (27) and (30)). 

Our approach was different. We first calculated all areas from geometrical considerations. It was easy to find the area of the set enclosed by the outer parallel curve $\Cp$, but it required some effort to calculate the areas of the sets enclosed by the loops of the inner parallel curve $\Cm$. Then, we derived the respective measures.

Clearly, the hitting probabilities in Corollary \ref{Cor:Corollary} remain unchanged if we throw a circle $C_0$ onto a lattice as in Fig.\ \ref{Fig:Reseau} where each circle is replaced by a congruent copy of an ellipse $C_1$.

\bigskip\noindent
Uwe B\"asel, HTWK Leipzig, Fakult\"at Maschinenbau und Energietechnik, PF 30 11 66, 04251 Leipzig, Germany, \texttt{uwe.baesel@htwk-leipzig.de}
\end{document}